\documentclass[12pt]{amsart}
\usepackage[utf8]{inputenc}
\usepackage{amsmath}
\usepackage{amsfonts}
\usepackage{amssymb}
\usepackage{amsthm}
\usepackage{IEEEtrantools}
 \usepackage{bbold}
\usepackage{mathrsfs}

\newcommand{\Nn}{\mathbb{N}}

\newcommand{\Rr}{\mathbb{R}}

\newcommand{\tp}{\mathrm{tp}}

\newcommand{\bdd}{\mathrm{bdd}}

\newcommand{\e}{\epsilon}

\def\Ind#1#2{#1\setbox0=\hbox{$#1x$}\kern\wd0\hbox to 0pt{\hss$#1\mid$\hss}
\lower.9\ht0\hbox to 0pt{\hss$#1\smile$\hss}\kern\wd0}

\def\notind#1#2{#1\setbox0=\hbox{$#1x$}\kern\wd0
\hbox to 0pt{\mathchardef\nn=12854\hss$#1\nn$\kern1.4\wd0\hss}
\hbox to 0pt{\hss$#1\mid$\hss}\lower.9\ht0 \hbox to 0pt{\hss$#1\smile$\hss}\kern\wd0}

\begin{document}

\newtheorem{lemma}{Lemma}[section]
\newtheorem{theorem}[lemma]{Theorem}
\newtheorem{definition}[lemma]{Definition}
\newtheorem{example}[lemma]{Example}
\newtheorem{corollary}[lemma]{Corollary}
\newtheorem{proposition}[lemma]{Proposition}
\newtheorem{claim}{Claim}[lemma]

\title{Model Theory of Scattered Piecewise Interpretable Hilbert Spaces}

\author{Alexis Chevalier}

\begin{abstract}
We study scattered piecewise interpretable Hilbert spaces from a model theoretic point of view. We establish strong connections between the Hilbert space structure theorems of \cite{ChevalierHru2021} and the model theoretic notions of canonical bases, one-basedness and model theoretic ranks.
\end{abstract}

\maketitle

Let $T$ be a complete continuous logic theory with a piecewise interpretable Hilbert space $\mathcal{H}$ as defined in \cite{ChevalierHru2021}. The main theorem of \cite{ChevalierHru2021}
  shows that a piecewise $\bigwedge$-interpretable subspace of $\mathcal{H}$ generated by a \emph{scattered} type-definable set is the orthogonal sum of piecewise $\bigwedge$-interpretable subspaces generated by \emph{asymptotically free} complete types. 
  
  Scattered piecewise interpretable Hilbert spaces were shown to capture a wide variety of  situations of interest in model theory and beyond, including $L^2$-spaces associated to absolute Galois groups or definable measures, unitary representations of $\omega$-categorical classical logic structures, and unitary representations of algebraic groups over local fields of characteristic $0$. In this paper, we show that these structure theorems and the surrounding notions are very natural from the  point of view of model theory.
  
  We establish a rough dictionary  between Hilbert space  notions and model theoretic notions. We connect \emph{weak closure} in Hilbert spaces and \emph{canonical bases} for Hilbert space types; \emph{scattered type-definable sets} and a local continuous logic version of \emph{U-rank};  \emph{commuting orthogonal projections} and \emph{one-basedness}. These connections are strong but they are not exact correspondences. 
 
%
%
%
%

In Section \ref{section: main definitions}, we  recall the main definitions from \cite{ChevalierHru2021}. In Section \ref{section: weak closure and canonical bases}, we  show that the partial order  on the weak closure of a type-definable set $p$ has a natural model theoretic interpretation in terms of forking relations between types. In Section \ref{section: ranks}, we study connections between scatteredness and the rank on the weak closure of a type-definable set. We show that the application of Von Neumann's lemma in \cite{ChevalierHru2021} is equivalent to showing one-basedness with respect to Hilbert space forking independence. We end by highlighting some open questions    and giving some examples of interpretable Hilbert spaces with surprising properties.
 
\section{Main definitions and notation}\label{section: main definitions}

For convenience, we recall some key notions from \cite{ChevalierHru2021}.   We will use the same notation. In this paper we always use continuous logic, as presented in \cite{Bekka2008} or in the appendix of \cite{ChevalierHru2021}.

Recall that for $M \models T$, a Hilbert space $H(M)$ is said to be piecewise interpretable in $M$ if $H(M)$ is a direct limit of imaginary sorts of $M$.  Since all data is assumed to be definable, any piecewise interpretable $H(M)$ in $M$ gives rise to a functor $\mathcal{H}$ which takes models of $T$ to Hilbert spaces such that for every $N \models T$, $H(N)$ is piecewise interpretable in $N$.

 We fix $\mathcal{H}$ piecewise interpretable in $T$.  
In this paper, we will move freely to imaginary sorts of $T$, so we can assume that the pieces of $\mathcal{H}$ belong to $T$ and that the direct limit maps on each piece of $\mathcal{H}$ are isometries. Therefore, $\mathcal{H}$ is in fact \emph{piecewise definable.}

For any $M \models T$, we will   identify vectors of $H(M)$ with elements of $M$. When $M \models T$ and $A \subseteq M$, we write $P_A$ for the orthogonal projection  onto the closed subspace of $H(M)$ densely spanned by $A\cap H(M)$.

Recall that a \emph{type-definable set in $\mathcal{H}$} is a type-definable set \emph{contained in a piece of $\mathcal{H}$}. If $p$ is type-definable in $\mathcal{H}$, we write $\mathcal{H}_p$ for the piecewise $\bigwedge$-interpretable subspace of $\mathcal{H}$ generated by $p$.

We recall some important notation from \cite{ChevalierHru2021}.

\begin{definition}
For $p$ a type-definable set in $\mathcal{H}$ and $M \models T$, we write $\mathcal{P}(p)$ for the weak closure of the set   of realisations of $p$ in $H(M)$. When $M \models T$ is $\omega_1$-saturated, we define the partial order $\leq_\mathcal{P}$ on $\mathcal{P}(p)$ as follows: for $u, v \in \mathcal{P}(p)$, $v \leq u$ if and only if there are $\bdd$-closed subsests $A_1, \ldots, A_n$ of $M$ such that $v = P_{A_1}\ldots P_{A_n} u$.
\end{definition}

 \cite{ChevalierHru2021} shows that $\mathcal{P}(p)$ is a type-definable set in a piece of $\mathcal{H}$.  
 Recall also from \cite{ChevalierHru2021} that when $M \models T$ is $\omega_1$-saturated, $\mathcal{P}(p)$ is the set of weak limit points of indiscernible sequences in $p$,  that $\mathcal{P}(p)= \{P_{\bdd(A)} a \mid a \models p, A \subseteq M\}$ and that $\mathcal{P}(p)$ is closed under the projections $P_{\bdd(A)}$ for $A \subseteq M$. 
 
 \begin{definition}
 We say that a type-definable set $p$ in $\mathcal{H}$ is scattered if for some (any) $M \models T$ $\omega_1$-saturated, $\mathcal{P}(p)$ is locally compact in $M$.
 
 We say that a type-definable set $p$ in $\mathcal{H}$ is asymptotically free if for any $M \models T$ and $a \models p$ in $M$, for any $\e > 0$, the type-definable set $p(x) \cup |\langle x, a\rangle| \geq \e$ is contained in $\bdd(a)$. Equivalently, for $M \models T$ $\omega_1$-saturated and $a,b \models p$ in $M$, we have either $\langle a, b\rangle = 0$ or $a \in \bdd(b)$. 
 \end{definition}
 
In the rest of this paper, we will work with $M \models T$, a type-definable set $p(x)$ in $\mathcal{H}$   and \emph{$\langle x,y \rangle$-types over $M$ consistent with $p$}. These are partial types $q$ of the form $ \{\langle x, b\rangle = \lambda(b) \mid b \in H(M)\}$ where $\lambda$ is a function $H(M) \to \Rr$. Note that this is technically a partial type in an infinite fragment of the language, as $y$ ranges over infinitely many sorts. In the next lemma, we show that we can always restrict  $y$ to range over  the same sort as $x$.

\begin{lemma}
Let $\mathcal{H}$ be piecewise interpretable in $T$ and let $M \prec N\models T$. Let $D$ be a piece of $\mathcal{H}$ and write $\mathcal{H}_D$ for the piecewise interpretable Hilbert space generated by $D$, so that $\mathcal{H}_D \leq \mathcal{H}$. Then $H_D(N)$ is orthogonal to $H(M)$ over $H_D(M)$. 
\end{lemma}

\begin{proof}
Let $v \in H(M)$. We need to show that $P_{H_D(N)} v = P_{H_D(M)} v$. It is enough to show that $d(v, H_D(N))$, the distance from $v$ to $H_D(N)$, equals $d(v, H_D(M))$. If $d(v, H_D(N)) = d$, then for any $\e$ there is a  formula which says that there exists $w_1, \ldots, w_{n_\e}$ in $D$ and scalars $\lambda_1, \ldots, \lambda_{n_\e}$ such that $\|v - \sum\lambda_i w_i\| < d + \e$. By elementarity, we can find such $w_i$ in $M$.
\end{proof}

Using the previous lemma, we lay down the convention that \emph{a $\langle x, y\rangle$-type over $M$ is a partial type in the unique formula $\langle x, y\rangle $ where $y$ is in the same sort as $x$}.

\section{Weak closure and canonical bases}\label{section: weak closure and canonical bases}

We fix $M \models T$   $\omega_1$-saturated and $p$ a type-definable set in $\mathcal{H}$.   Let $b \in \mathcal{P}(p)$ and let $(a_n)$ be a sequence in $p$ converging weakly to $b$. Define  $p^b(x)$ to be the $\langle x, y\rangle$-type over $M$  defined by $\langle x, c \rangle = \lim_n \langle a_n, c\rangle = \langle b, c\rangle$.

\begin{lemma}\label{lemma: weak closure is set of canonical bases}
For any $M \models T$ $\omega_1$-saturated, $\mathcal{P}(p)$ is a set of canonical bases for $\langle x, y \rangle$-types over $M$ consistent with $p$, where $b$ is  a canonical base for $p^b$. 
\end{lemma} 

\begin{proof}
$p^b(x)$ is a $\langle x, y \rangle$-type over $H(M)$ consistent with $p$. Conversely, let $q$ be any $\langle x, y \rangle$-type over $H(M)$ consistent with $p$. Let $M \prec N$ and let $a$ be a realisation of $q(x) \cup  p(x)$ in $N$. Let $c = P_{H(M)}(a)$. Then for all $v \in H(M)$, $\langle a, v\rangle = \langle c, v\rangle$, so $q$ is definable over $c$. We only need to check that $c \in \mathcal{P}(p)$, so that $q = p^c$.

By $\omega_1$-saturation of $M$, find $a' \in M$  a realisation of $p \cup (q \upharpoonright \bdd(c))$. Then $P_{\bdd(c)} a' = c$ but we already know that $P_{\bdd(c)} a' \in \mathcal{P}(p)$, so we are done.
\end{proof}

\noindent \textbf{Remark:} We are making a slightly unconventional use of the term `canonical base'. Canonical bases are usually dcl-closed sets. See for example \cite{Pillay1996} for a classical discussion of canonical bases. In Lemma \ref{lemma: weak closure is set of canonical bases}, nothing ensures that distinct $b, b' \in \mathcal{P}(p)$ are not inter-definable. Nevertheless, we will say that $b$ is `the' canonical base of $p^b$. With this choice of terminology, for any $b \neq b' \in \mathcal{P}(p)$, $b$ and $b'$ are canonical bases for different types over $M$ even though they may be inter-definable.

\medskip

The next lemma follows directly by applying Lemma \ref{lemma: weak closure is set of canonical bases} with the type-definable set $\mathcal{P}(p)$ instead of $p$. 

 
\begin{lemma}\label{lemma: P(p) has built-in canonical bases}
 $\mathcal{P}(p)$ has built-in canonical bases for $\langle x, y\rangle$-types over models consistent with $p$: for any $M \models T$,  for any $\langle x, y\rangle$-type $q(x)$ over $M$ consistent with $\mathcal{P}(p)$, there is a unique $b \in \mathcal{P}(p)$ in $M$ such that $q(x)$ is defined over $b$ by $\langle x, v\rangle = \langle b, v\rangle$ for every $v \in H(M)$. Therefore $q = p^b$.
\end{lemma}
 
 We interpret Lemma \ref{lemma: P(p) has built-in canonical bases} as saying that the weak closure of a type-definable set is a natural object from a model theoretic point of view.

%
%

We now define a partial order which gives model-theoretic significance to the partial order $\leq_\mathcal{P}$ on $\mathcal{P}(p)$. 
 
\begin{definition}\label{definition: model theoretic partial order}
Define the relation $L_1(x, y)$ on $\mathcal{P}(p) \times \mathcal{P}(p)$ in $M$ by saying that $L_1(b, c)$ if and only if $p^b$ extends $p^c \upharpoonright \bdd(c)$. Define the relation $\leq_1$ on $\mathcal{P}(p) \times \mathcal{P}(p)$ as the transitive closure of $L_1$.

\end{definition} 
 
For any $  c \in \mathcal{P}(p)$, $p^c$ is the unique nonforking extension of $p^c \upharpoonright \bdd(c)$ to $M$, since $\langle x, y\rangle$ is a stable relation (see \cite{BenYaacov2010} for a discussion of stability in continuous logic). Therefore, $L_1(b, c)$ and $b\neq c$ if and only if $p^b$ is a forking extension of $p^c \upharpoonright \bdd(c)$.

\begin{lemma}\label{lemma: partial orders anti-isomorphic}
For any $b, c \in \mathcal{P}(p)$,  $L_1(b, c)$ if and only if there exists a small $\bdd$-closed $A\subseteq M$ such that $c = P_A b$. Hence the partial order $\leq_1$ is anti-isomorphic to $\leq_\mathcal{P}$.
\end{lemma}

\begin{proof}
For any small $\bdd$-closed $A$, $c = P_A b$ if and only if $c = P_{\bdd(c)}b$ if and only if $p^b$ extends $p ^c \upharpoonright \bdd(c)$.
\end{proof}



In the previous lemma, we characterised $\leq_\mathcal{P}$ in terms of forking extensions of $\langle x, y\rangle$-types associated to elements of $H(M)$. Next we show that the partial order $\leq_\mathcal{P}$ can also be seen as a natural partial order between the $\langle x, y\rangle$-types themselves. 

In the following definition and lemma, a $\langle x, y\rangle$-type over a $\bdd$-closed subset of $M$ is a partial type in the single function $\langle x, y\rangle$ where $y$ ranges in the same sort as $x$ over a set $A \subseteq M$ such that $A$ is $\bdd$-closed \emph{in the sort of $x$.}

\begin{definition}\label{definition: second partial order}
Define the relation $L_2$ on $\langle x, y\rangle$-types over   $\bdd$-closed subsets of $M$ consistent with $\mathcal{P}(p)$ as follows: if $r_1$ is over $B$ and $r_2$ is over $C$, then $L_2(r_1, r_2)$ if and only if $(r_1\upharpoonright M) \upharpoonright C = r_2$, where $r_2\upharpoonright M$ is the unique nonforking extension of $r_1$ to $M$. Define the relation $\leq_2$ between types as the transitive closure of $L_2$.

%
\end{definition}

For the next lemma, recall that types $r_1, r_2$ over $B, C$ respectively are said to be \emph{parallel} if the nonforking extensions $r_1 \upharpoonright M$ and $r_2 \upharpoonright M$ are equal.

\begin{lemma}\label{lemma: third partial order coincides}
Let $B, C$ be   $\bdd$-closed subsets of $M$. Let $r_1$, $r_2$ be $\langle x, y\rangle$-types over $B, C$ respectively consistent with  $p$ with canonical parameters $b, c$ in $\mathcal{P}(p)$. Then $L_2(r_1, r_2)$ if and only if $L_1(b, c)$. 

Hence $\leq_2$ induces a partial order on parallelism classes of $\langle x, y\rangle$-types over small $\bdd$-closed subsets of $M$ consistent with $p$. This partial order is anti-isomorphic to $\leq_\mathcal{P}$.

\end{lemma}

\begin{proof}
If $A \subseteq M$ is a   $\bdd$-closed subset and $r$ is a $\langle x, y\rangle$-type over $A$ consistent with $p$, then $r$ has a canonical base $b \in A$ and Lemma \ref{lemma: P(p) has built-in canonical bases} shows that $r = p^b \upharpoonright A$.

Therefore $L_2(r_1, r_2)$ if and only if $p^b \upharpoonright C = r_2$ if and only if  for all $v \in C$, $\langle b, v\rangle = \langle c, v\rangle$ if and only if $P_C b = c$. By Lemma \ref{lemma: partial orders anti-isomorphic}, this is equivalent to $L_1(b, c)$.

$r_1, r_2$ are parallel if and only if $r_1\upharpoonright M = r_2 \upharpoonright M$ if and only if $p^b = p^c$ if and only if $b = c$, by Lemma \ref{lemma: P(p) has built-in canonical bases}. Hence $\leq_2$ induces a partial order on parallelism classes and these parallelism classes are in an obvious bijection with $\mathcal{P}(p)$.
\end{proof}

\section{Ranks in scattered interpretable Hilbert spaces}\label{section: ranks}

In this section, we fix $M \models T$ a large saturated model and $p$ a type-definable set in $\mathcal{H}$.  Recall the following theorem from \cite{ChevalierHru2021}:

\begin{theorem}[\cite{ChevalierHru2021}]\label{projections commute}
Let $p$ be a scattered type-definable set in $\mathcal{H}$.   Let $A$, $B$ be $\bdd$-closed subspaces of $H(M)$. Then $A \cap H_p(M)$ and $B \cap H_p(M)$ are orthogonal over $A \cap B \cap H_p(M)$. Equivalently,  for any $v \in A \cap H_p(M)$, we have $P_{B}v = P_{A \cap B}v$. Equivalently, for any $v \in H_p(M)$,
\[
 P_{B} P_{A}v = P_{A}P_{B}v = P_{A \cap B}v\]
\end{theorem}

The proof of Theorem \ref{projections commute} hinges on an application of Von Neumann's lemma (see \cite{VonNeumann1950}). We refer the reader to \cite{ChevalierHru2021} for the proof. 
We show that the conclusion of Theorem  \ref{projections commute} is equivalent to a local form of one-basedness.

\begin{definition}
Let $A $ be a small $\bdd$-closed subspace of $H(M)$ and let $q(x)$ be a  $\langle x, y\rangle$-type  over $A$ consistent with $p$. We say that $q$ is one-based in $p$ if for any realisation $a$ of $q \cup p$ in $ M$, $q$ is definable over $A \cap \bdd(a) \cap H_p(M)$. 


 \end{definition}

See \cite{Pillay1996} for a general treatment of one-basedness in classical logic. 

\begin{lemma}
$\langle x, y\rangle$-types over $\bdd$-closed subsets of $M$ consistent with $p$ are one-based in  $p$ if and only for all $\bdd$-closed $A, B \subseteq M$, $P_AP_B = P_{A \cap B }$ on $H_p(M)$.
\end{lemma}

\begin{proof}

Suppose first that $\langle x, y \rangle$-types consistent with $p$ are one-based in $p$ and take $A, B \subseteq M$ $\bdd$-closed. Let $a$ be any realisation of $p$ and let $b = P_Aa \in A$. $b$ is the canonical parameter of the $\langle x, y\rangle$-type of $a$ over $A$. Then $P_BP_A a = P_Bb$,  which is the canonical parameter of the $\langle x,y \rangle$-type of $b$ over $B$.
 By one-basedness  $P_B b \in B \cap \bdd(b) \subseteq A \cap B$. Therefore $P_BP_Aa  \in A \cap B$. By the same computation as in the proof of Theorem \ref{projections commute}, we deduce $P_Aa \in A \cap B$ and hence $P_B P_Aa = P_{A \cap B}$ on $H_p(M)$.



Conversely,   let $q$ be a $\langle x, y\rangle$-type over a $\bdd$-closed $A \subseteq M$ consistent with $p$. Let $a$ be a realisation of $q \cup p$ in $M$. Let $b = P_A a$ be the canonical parameter of $q$. Then $b = P_A P_{\bdd(a)}a \in A \cap \bdd(a)$. 
\end{proof}

\noindent \textbf{Remark:} In \cite{ChevalierHru2021}, one-basedness in interpretable Hilbert spaces was found to be the key technical fact which made it possible to derive structure theorems for scattered interpretable Hilbert spaces. Indeed, using one-basedness, it is easy to prove that when $p$ is scattered, $(\mathcal{P}(p), \leq_\mathcal{P})$ is well-founded, and this allows a decomposition of $\mathcal{H}_{\mathcal{P}}$ into subspaces generated by asymptotically free sets. Recall also from \cite{ChevalierHru2021} that the structure theorem for scattered interpretable Hilbert spaces   eventually leads us to the classification of unitary representations of oligomorphic groups given in \cite{Tsankov2012}. 

Conversely, it is easy to  see directly that if $p$ is an asymptotically free type-definable set then $\langle x, y\rangle$-types over $\bdd$-closed sets are one-based. Therefore, in the case of classical logic $\omega$-categorical theories, one-basedness follows from the classification theorem of \cite{Tsankov2012}. This has been well-known to several authors since Tsankov's original result. 
 In the $\omega$-categorical setting, the paper \cite{BenYaacov2017} studies  conditions under which one-basedness in interpretable Hilbert spaces gives rise to full one-basedness for stable independence in the  classical logic $\omega$-categorical setting. See Theorem 3.12 in \cite{BenYaacov2017}.

\medskip

  We now turn to a general discussion of properties of the partial order $\mathcal{P}(p)$.

 \begin{definition}
If $(\mathcal{Q}, \leq_\mathcal{Q})$ is an arbitrary partial order, we define the foundation rank $F_\mathcal{Q}(x)$ of $x \in \mathcal{Q}$ as follows:
\begin{enumerate}
\item $F_\mathcal{Q}(x) \geq 0$ for all $x$
\item $F_\mathcal{Q}(x) \geq \lambda$ for limit ordinal $\lambda$ if $F(x) \geq \alpha$ for all $\alpha < \lambda$
\item $F_\mathcal{Q}(x) \geq \alpha + 1$ if there is $y <_\mathcal{Q} x$ such that $F_\mathcal{Q}(y) \geq \alpha$
\end{enumerate}
We say that $F_\mathcal{Q}(x) = \infty$ if $F_\mathcal{Q}(x) \geq \alpha$ for every ordinal $\alpha$ and   $F_\mathcal{Q}(x) = \alpha$ if $F_\mathcal{Q}(x) \geq \alpha$ and $F_\mathcal{Q}(x) \not \geq \alpha + 1$. 

We write $F_\mathcal{Q}(\mathcal{Q}) = \sup\{ F_\mathcal{Q}(x) \mid x \in \mathcal{Q}\}$.
\end{definition}

\begin{proposition}\label{proposition: ordinal foundation rank gives one-based}
Let $F_\mathcal{P}$ be the foundation rank of $(\mathcal{P}(p), \leq_\mathcal{P})$. If for every  $x \models p$ we have $F_{\mathcal{P}}(x) < \infty$, then $\langle x, y\rangle $-types consistent with $p$ are one-based in $p$.  
\end{proposition}

\begin{proof}
 This is similar to the argument behind Theorem \ref{projections commute}. For any $\bdd$-closed $A, B  \subseteq M$, for any $a \models p$, the sequence of alternating projections $P_AP_B \ldots P_A P_Ba$ must eventually be constant, since $F_\mathcal{P}$ is ordinal-valued. The argument of Theorem \ref{projections commute} shows that we must then have $P_A P_B = P_{A \cap B}$ so $\langle x, y\rangle$-types are one-based.
\end{proof}

In light of the anti-isomorphism between $\leq_\mathcal{P}$ and the partial orders $\leq_1$ and $\leq_2$, it is natural to define a local continuous logic version of $U$-rank which captures forking inside the Hilbert space. In the next definition, we call this the $V$-rank. The reader is referred to \cite{Pillay1996} for a general discussion of  U-rank in classical logic. 

 \begin{definition}
For any $\bdd$-closed subset $A$ of $M$ and $q$ a $\langle x, y\rangle$-type over $A$ in $\mathcal{H}$, define the relation $V(q) \geq \alpha$ for $\alpha$ an ordinal as follows:

 \begin{enumerate}
\item $V(q) \geq 0$ for all $q$
\item $V(q) \geq \lambda$ for limit ordinal $\lambda$ if for every $\alpha < \lambda$, $V(q)\geq \alpha$
\item $V(q)\geq \alpha + 1$ if there is some    $B\supseteq A$ and  $q'$ over $B$  extending $q$ such that $V(q') \geq \alpha$ and $q'$  forks over $A$ with respect to the function $\langle x, y \rangle$. 
 \end{enumerate}
If $V(q) \geq \alpha$ for all $\alpha$, we say $V(q) = \infty$. We say $V(q) = \alpha$ if $V(q) \geq \alpha$ and $V(q) \not \geq \alpha+1$. 

For $a \in H(M)$, write $V(a/A)$ for the $V$-rank of the $\langle x, y\rangle$-type of $a$ over $A$.  
 \end{definition}

 \begin{lemma}\label{lemma: V rank as foundation}
Let $A$ be a small $\bdd$-closed subspace of $H(M)$ and let $q$ be a $\langle x, y\rangle$-type over $A$ consistent with $p$. Let $b$ be the canonical parameter of $q$ in $\mathcal{P}(p)$. Then $V(q)$ is equal to the foundation rank of $b$ in $(\mathcal{P}(p), \leq_1)$.
 
\end{lemma}

\begin{proof}
Note that $b \in A$. We have already observed that for any $c \in \mathcal{P}(p)$, $L_1(b, c)$ and $b \neq c$ if and only if $p^b$ is a forking extension of $p^c \upharpoonright \bdd(c)$. The lemma follows easily.
\end{proof}

Recall from \cite{ChevalierHru2021} that if $p$ is a type-definable set in an interpretable Hilbert space $\mathcal{H}$, we say that the inner product is \emph{strictly definable on $p$} if the set of scalars $\{ \langle x, y \rangle \mid x, y \models p\}$ is finite. We show that in this case,  the $V$-rank coincides with the Shelah $\omega$-local rank defined in \cite{Shelah1978a}. We recall the definition of this rank here:

\begin{definition}[\cite{Shelah1978a}, II.1.1]\label{definition: shelah rank}
In an arbitrary theory $T$, let $\Delta(x, y)$ be a set of   formulas and $p(x)$ a type-definable set, possibly with parameters. For $\alpha$ an ordinal, we define $R_\Delta(p) \geq \alpha$ as follows:
\begin{enumerate}
\item $R_\Delta(p)\geq 0$ if $p$ is consistent
\item For $\lambda$ a limit ordinal $R_\Delta(p)\geq \lambda$ if $R_\Delta(p)\geq \alpha$ for all $\beta < \alpha$
\item $R_\Delta(p)\geq \alpha +1$ if for every finite $p' \subseteq p$ and every $n < \omega$ there are $\Delta$-types $(q_m(x))_{m < n}$ such that:
\begin{enumerate}
\item for $m \neq m' < n$, $q_m(x) \cup q_{m'}(x)$ is inconsistent
\item $R_\Delta(p' \cup q_m)\geq \alpha$ for all $m$.
\end{enumerate}
\end{enumerate}
We write $R_\Delta(p) = \alpha$ if $R_\Delta(p)\geq \alpha$ and $R_\Delta(p)\not\geq \alpha+1$.
\end{definition}

Since the above definition is borrowed from \cite{Shelah1978a}, we chose to keep the term `formula'. Since we are working in continuous logic, the term `condition' might be more appropriate, to use the formalism of \cite{BenYaacov2008}. Alternatively, the formalism given in the appendix of \cite{ChevalierHru2021} keeps the term `formula'. In any case, when the inner product is strictly definable, any expression of the form $\langle x, y\rangle = \lambda$ behaves like a classical logic formula, since it admits negation. Therefore in the next proposition, we will draw on results from classical logic and we will keep our terminology in line with this.

\begin{proposition}\label{proposition: finite rank in strictly definable hilbert space}
Suppose $S$ is a piece of $\mathcal{H}$ such that the inner product  is strictly definable on $S$.  Let $\Delta$ be the finite set of formulas $\langle x, y \rangle  = \lambda$ where $x, y$ range in $S$. Then for every $a \in \mathcal{P}(S)$ and $A \subseteq M$ $\bdd$-closed, $V(a/A) = R_\Delta(a/A) < \omega$. 
\end{proposition}

\begin{proof}
This is proved entirely by using results of \cite{Shelah1978a}. We sketch the proof here for convenience. Let $q$ be a $\langle x,y \rangle$-type over $A$ consistent with $p$. Define the rank $R^*_\Delta(p)$ in the same way as in Definition \ref{definition: shelah rank} except that we say $R^*_\Delta(q) \geq \alpha+1$ if there are $\Delta$-types $(q_i)_{i < \omega}$ which are pairwise inconsistent and $R_\Delta^*(q \cup q_i) \geq \alpha$ for all $i$.

Now all references are from \cite{Shelah1978a}. In general for $\Delta$ a finite set of stable formulas, we can assume that $\Delta$ is a single stable formula by II.2.1. By II.2.7, $R_\Delta(q) < \omega$. We see  from definitions that $R_\Delta^*(q) \leq R_\Delta(q)$. By II.2.9, the property $R_\Delta(q) \geq n$ is type-definable. It then follows by a compactness argument that $R^*_\Delta(q) \geq R_\Delta(q)$. Therefore $R_\Delta(q) = R^*_\Delta(q)$. 

Now it follows directly from the definition of $R^*_\Delta$ that if $R_\Delta(q) = n+1$,   there is a $\Delta$-type $q'$ over some $B \supseteq A$ such that $R_\Delta(p \cup q) = n$.  Furthermore, II.1.2 and III.4.1 show that $q'$ does not fork over $A$ if and only if $R_\Delta(q') = R_\Delta(q)$. The proposition follows.
\end{proof}

In the remainder of this paper, we will be interested in understanding the relations between scatteredness, one-basedness, the $F_\mathcal{P}$-rank, and the $V$-rank.  In \cite{ChevalierHru2021}, it is shown that if $p$ is scattered, then $(\mathcal{P}(p), \leq_\mathcal{P})$ is well-founded so the $F_\mathcal{P}$-rank is ordinal valued. We record   some other connections between these different notions.

 \begin{proposition}\label{proposition: scattered implies V rank}
If  $p$ is scattered then for every $x \in \mathcal{P}(p)$, $V(x) < \infty$. Equivalently, if $p$ is scattered, there is no infinite $\leq_\mathcal{P}$-increasing sequence in $\mathcal{P}(p)$.
 \end{proposition}
 
 \begin{proof}
  Suppose   that $p$ is scattered and that $(a_n)$ is an infinite increasing sequence in $ \mathcal{P}(p)$. Then $a_n = P_{\bdd(a_n)} a_{n+1}$ for all $n$. By one-basedness, $\bdd(a_n) \subseteq \bdd(a_{n+1})$ for all $n$. Write $V_n$ for the subspace $\bdd(a_n)$ and $W_n$ for the orthogonal complement of $V_n$ in $V_{n+1}$. Then for all $n$, $a_{n+1} = a_n + w_n$ where $w_n \in W_n$ is orthogonal to $a_n$. Hence $a_n = a_0 + \sum_{i = 0}^{n-1} w_i$. 

Since every $a_n$ is in the type-definable set $\mathcal{P}(p)$, $\|a_n\|$ is bounded above and the sequence $(\|a_n\|)$ is convergent. Now $\|a_n\|^2 = \|a_0\|^2 + \sum_{k = 0}^{n-1}\|w_k\|^2$ since the vectors $(w_k)$ and $a_0$ are pairwise orthogonal. For $n \geq m$, 
\[
\|a_n - a_m\|^2 = \|\sum_{k = m}^{n-1}w_n\|^2 = \sum_{k = m}^{n-1}\|w_n\|^2 = \|a_n\|^2 - \|a_m\|^2
\]
so $(a_n)$ is Cauchy and hence convergent to some $a \in \mathcal{P}(p)$. Note that  $a_n \leq_{\mathcal{P}} a$ for all $n$. Let $\e > 0$ and take $n$ such that $\|a -a_n\| < \e$. There is an indiscernible sequence $(b_k)$ in $\tp(a/\bdd(a_n))$ starting at $a$ and converging weakly to $a_n$. For all $k$, $\|b_k - a\|^2 = 2\|a\|^2 - 2 \langle b_k, a\rangle = 2\|a\|^2 - 2\langle a_n, a\rangle$. Choosing $\e$ small enough, we find that $(b_k)$ is arbitrarily close to $a$ and hence $\mathcal{P}(p)$ is not scattered, a contradiction.
 \end{proof}

 The next proposition studies a case where scatteredness fails in a strong way. 
 
 \begin{proposition}
If $p$ is a complete type and is not locally compact, then $F_{\mathcal{P}}(a) \geq \omega$ for   $a \models p$. 
\end{proposition} 

\begin{proof}
Take $a \models p$ in $M$. Note that $a \neq 0$ and $\{b \in \mathcal{P}(p)\mid b < a\}$ is nonempty. Suppose that we have $b_0 <_\mathcal{P} \ldots  <_\mathcal{P}b_n <_{\mathcal{P}} a $ in $\mathcal{P}(p)$. Suppose that $\|a - b_n\|  = \e$. Since $p$ is not locally compact around  $a$, we can find an infinite indiscernible sequence $(a_k)$ with $a_0 = a$ such that $\|a_k - a_j\| = \delta < \e$ for all $k \neq j$, for arbitrarily small $\delta$. Write $c$ for the weak limit of $(a_k)$. Then $\|a - c\| < \|a - b_n\|$ so $c\neq b_n$. Since $b_n = P_{\bdd(b_n)}a$ and $c = P_{\bdd(c)}a$, we see that $c \notin \bdd(b_n)$. Moreover, $\|P_{\bdd(b_n)} c - b_n\| \leq \|c - a\| = \delta$, so we can assume that $P_{\bdd(b_n)}c \notin \bdd(b_{n-1})$ by taking $\delta$ small enough. 

Taking into account all distances $\|b_{k+1} - b_k\|$, choosing $\delta$ small enough and writing $c_{n+1} = c$, $c_i = P_{\bdd(b_i)} c_{i+1}$, we have $c_0 <_\mathcal{P} \ldots <_\mathcal{P} c_{n+1} <_\mathcal{P} a$. This proves that $F_\mathcal{P}(a) \geq \omega$.
\end{proof}

%
%

 \medskip

%
%
 
\begin{proposition}\label{proposition: bad examples}
There are examples of theories $T$ with interpretable Hilbert spaces $\mathcal{H}$ and type-definable sets $p$ satisfying any of the following:
\begin{enumerate}
 \item $p$ is not locally compact,  $(\mathcal{P}(p), \leq_\mathcal{P})$ has foundation rank $\omega$ but $(\mathcal{P}(p), \leq_1)$ is not well-founded
 \item $p$ is not locally compact,  $(\mathcal{P}(p), \leq_1)$ has foundation rank $\omega$, $(\mathcal{P}(p), \leq_\mathcal{P})$ is not well-founded, every type in $\mathcal{P}(p)$ is locally compact and $\langle x, y\rangle$-types consistent with $p$ are one-based in $p$. 
 \item $p$ is not locally compact, $\langle x, y\rangle$-types consistent with $p$ are one-based in $p$ but $(\mathcal{P}(p), \leq_\mathcal{P})$ and $(\mathcal{P}(p), \leq_1)$ are not well-founded.
 \end{enumerate}
\end{proposition} 

\begin{proof}
(1) Let $T$ be the classical logic theory of a collection of equivalence relations $E_n$ on a set $X$ such that $E_0(x, y)$ is the trivial equivalence relation $x = x$ and $E_{n+1}$ refines each $E_n$-class into infinitely many infinite $E_{n+1}$-classes. Let $M\models T$ and for every $n \geq 0$ let $(\alpha_k^n)_{k < \kappa}$ be an enumeration of the $E_n$-classes of $M$. Let $H$ be the free Hilbert space on the set $\bigcup_{k, n \geq 0} \alpha_k^n$, so that $\{\alpha_k^n\mid k, n\geq 0\}$ is an orthonormal basis of $H$. 

Define $h : X \to H$ by $h(x) = \sum_{n \geq 0} \alpha^n_{k(x)}/2^n$ where  $\alpha^n_{k(x)}$ is the $E_n$-class of $x$. Then $h$ gives rise to an interpretation of $H$ in $M$. Let $p$ be the quotient of $X$ by the equivalence relation $h(x) = h(y)$. We view $p$ as a subset of $h$. Then $\mathcal{P}(p) =  p \cup \{\sum_{n \leq m} \alpha_{k(x)}^n/2^n \mid  x \in X, m \geq 0\} $.

(2)   Let $T$ be the classical logic theory of a collection of equivalence relations $(E_n)$ on a set $X$ such that $E_0(x, y)$ is the  equivalence relation $x = y$ and each $E_{n+1}$-class is an infinite union of $E_{n}$-classes.
  Let $M\models T$ and for every $n \geq 0$ let $(\alpha_k^n)_{k < \kappa}$ be an enumeration of the $E_n$-classes of $M$. Let $H$ be the free Hilbert space on the set $\bigcup_{k, n \geq 0} \alpha_k^n$.
  
   Define $h : X \to H$ by $h(x)=  \sum_{n \geq 0} \alpha_{k(x)}^n/2^n$ where $\alpha_{k(x)}^n$ is the $E_n$-class of $x$. Note that $h$ is injective and that the set $h(X)$ is discrete in $H$, since $x \neq y$ implies $\|x - y\| \geq 1$. We also have $\mathcal{P}(X) = \{0\} \cup \{\sum_{n \geq m} \alpha_{k(x)}^n/2^n \mid x \in X, m \geq 0\}$. $\mathcal{P}(X)$ is locally compact everywhere except around $0$.

(3) Let $T$ be the classical logic theory of an infinite set $X$. Let $\mathcal{H}$ be the interpretable Hilbert space such that $X$ is an orthonormal basis of $\mathcal{H}$. Let $M \models T$ be $\omega_1$-saturated as a continuous logic theory. Let $(e_n)$ be an infinite sequence in $X$ and let $p$ be the complete type of the vector $\sum_{n \geq 0} e_n/2^n$. Then $\mathcal{P}(p)=  \{\sum_{n \in J} e_n/2^n\mid J \subseteq \Nn\}$ but $\mathcal{H}$ is one-based because $X$ is scattered.
\end{proof}
 
 The next lemma shows that the rank $\mathcal{F}_{\mathcal{P}}$ is not usually definable. 
\begin{lemma}
 There exists a classical logic theory $T$ with an interpretable Hilbert space $\mathcal{H}$ and a piece $S$ such that 
$S$ is scattered, $\mathcal{F}_\mathcal{P}(S) = \omega$, and  for every complete type $p$ in $S$, $\mathcal{F}_\mathcal{P}(p) < \omega$. 
\end{lemma}
 
\begin{proof}
Let $\mathcal{L}$ be the language of a family $(E_n)_{n \geq 1}$ of binary relations. Consider the $\mathcal{L}$-structure $M$ satisfying the following:
\begin{enumerate}
\item every relation $E_n$ is an equivalence relation such that every $E_n$-class is infinite
\item $E_1$ has countably many classes, which we write $(A_n)_{n \geq 1}$, all of which are infinite
\item Fix $m \geq 1$. Then for every $1 \leq k < m$, $E_{k+1}$ refines $E_k$ on $A_m$ into infinitely many equivalence classes. For every $k \geq m$, $E_{k +1} = E_k$ on $A_m$.
\end{enumerate}
Define the function $f$ on $M \times M$ by 
\begin{enumerate}
\item if $x,y$ are not $E_n$-equivalent for any $n$, then $f(x, y) = 0$. If $x, y$ are $E_n$-equivalent for all $n$, then $f(x, y) = 1$.
\item  if $x,y$ are both in $A_m$ but not $E_n$-equivalent for all $n$, then let 
\[
k = \min(\max\{n \geq 1\mid x\text{ and } y\text{ are }E_n\text{-equivalent}\}, m).
\]
We set
\[
f(x, y) = \frac{m-1}{m}+ \frac{1}{m(m+1 - k)}.
\]
\end{enumerate}
The exact numerical values of $f$ are not essential for our argument. Note that as $m \to \infty$, $f(x, y) \to 1$ on $A_m$ and that on each $A_m$, $f$ takes $m$ distinct values.

In a similar way to the examples given in  Proposition \ref{proposition: bad examples}, one can show that  $f$ is also positive semidefinite. $f$ is clearly definable in $T = Th(M)$. Let $\mathcal{H}$ be the interpretable Hilbert space in $T$ generated by $f$. 

Note that $M$ realises all type of $S$ except one, call it $p$.  $f(x, y) = 1$ on $p \times p$, so $\mathcal{F}_{\mathcal{P}}(p) = 0$.

Now let $q$ be a complete type realised in $M$. Then $q$ is contained in some definable set $A_m$ and it is easy to see that $\mathcal{F}_{\mathcal{P}}(q) = m$.  Therefore $\mathcal{F}_{\mathcal{P}}(S) = \omega$. 
 \end{proof}

\noindent \textbf{Open Question}: Is it possible to find $T$ and $\mathcal{H}$ with a scattered type-definable set $p$ such that $\omega < \mathcal{F}_\mathcal{P}(p) < \infty$?

\bibliographystyle{alpha}
\bibliography{rank-paper}

\end{document}